\documentclass[11pt]{amsart}

\usepackage[article]{robertpreamble}
\usepackage[margin=1in]{geometry}

\newcommand{\Rea}{\operatorname{Re}}

\begin{document}


\title{Nested nodal loops of biharmonic functions}


\author{Javier Gómez-Serrano}
\address{Department of Mathematics, Brown University}
\email{javier\_gomez\_serrano@brown.edu}

\author{Robert Koirala}
\address{Department of Mathematics, University of California San Diego}
\email{rkoirala@ucsd.edu}

\author{Alexander Logunov}
\address{Department of Mathematics, Massachusetts Institute of Technology}
\email{alogunov@mit.edu}
\date{\today}

\begin{abstract}
   Given any \(n\in\mathbb{N}\), we construct a real-valued biharmonic polynomial on \(\mathbb{R}^2\) whose zero set contains a nest of \(n\) smooth, disjoint topological loops, meaning that the \(k\)-th loop lies inside the domain bounded by the \((k+1)\)-st loop for \(k=1,\ldots,n-1\). The case \(n=2\), i.e., the existence of two nested loops, is related to the failure of the Boggio-Hadamard conjecture from the early 1900s.
\end{abstract}

\maketitle 


\section{Introduction}
 
Consider a real-valued biharmonic function \(u\) on \(\R^2\), that is,
\begin{align}
    \Delta \Delta u=0
\end{align}
and assume that the zero set of $u$ has no singular points, i.e., $u(x)=0$ implies that $|\nabla u(x)|\neq0$.

Recently, Nadirashvili and the third author asked whether the nodal set \(u^{-1}(0)\) can contain a double nest of zero loops, meaning that one zero loop bounds a domain containing another zero loop, as in
\begin{tikzpicture}[baseline=-0.5ex,scale=0.7]
    \draw (0,0) circle (0.25);
    \draw (0,0) circle (0.10);
\end{tikzpicture}.

An example of a biharmonic polynomial with double nest of zero (nodal) loops was constructed in \cite{Koirala-inprep}. On the other hand, it appears that the question is related to the conjecture of Boggio--Hadamard \cite{Boggio1901, Boggio1905, Hadamard1908} on the positivity of Green's function for the bilaplacian \(\Delta \Delta\) in a simply connected domain with clamped boundary conditions.

Reformulation of the conjecture in the form of a maximum principle, see \cite{MR1696190}, states that if \(D\subset\mathbb{R}^2\) is a bounded simply connected domain with sufficiently smooth boundary and \(v\in C^4(D)\cap C^1(\overline D)\), then
\begin{align*}
    \Delta^2 v \leq 0 \quad \text{in } D,\qquad
    v\leq 0 \quad \text{on } \partial D,\qquad
    \frac{\partial v}{\partial n}\leq 0 \quad \text{on } \partial D
\end{align*}
imply
\begin{align*}
    v\leq 0 \quad \text{in } D.
\end{align*}
Here \(\partial/\partial n\) denotes differentiation in the inward normal direction.

The Boggio-Hadamard conjecture was posed in early 1900s, but was disproved almost 50 years later by Duffin and Garabedian, with an elementary counterexample later given by Shapiro and Tegmark, see \cite{MR29021, MR46440, MR1267051,MR2591975,MR745128}. Nevertheless, let us wrongly assume that the conjecture is true and let us show why it implies that $u$ cannot have a double nest. 

Suppose that \(u\) has a double nest consisting of two smooth loops \(\Gamma_1\) and \(\Gamma_2\), where \(\Gamma_2=\partial D\) for some simply connected domain \(D\), and \(\Gamma_1\subset D\). By assumption, \(\nabla u\) does not vanish on \(\Gamma_2\). After replacing \(u\) by \(-u\), if necessary, we may assume that \(\frac{\partial u}{\partial n}\leq 0 \quad \text{on } \partial D,\) also $u$ vanishes on $\partial D$ and biharmonic in $D$. The Boggio-Hadamard conjecture implies that $u \leq 0$ in $D$, which contradicts to the fact that $u$ changes sign near $\Gamma_1$.

\begin{theorem}\label{thm:main}
    For every integer $n\ge 1$, there exists a non-zero biharmonic polynomial $u_n:\R^2\to\R$ such that the nodal set $u_n^{-1}(0)$ contains $n$ nested smooth loops.
\end{theorem}

Let us note that within biharmonic polynomials of degree $d$ our construction gives only a nest of zero loops of depth $[c\log d]$ for some numerical $c>0$. For general real polynomials of two variables of degree $d$ (without biharmonicity) the maximal depth of a nest of the zero loops is at most $[d/2]$, which follows (Hilbert's observation from \cite{Hilbert1891}) by drawing a line through the center of the nest and using the fact that one-dimensional non-zero polynomial of degree $d$ cannot have more than $d$ roots; there are elementary examples of polynomials of even degree $d$ with a nest of depth $[d/2]$ such as 
\begin{align*}
    P(x,y)= \prod\limits_{k=1\dots [d/2]} (x^2+y^2 - k^2).
\end{align*}

In view of this, it seems that in the regime when the degree $d\to +\infty$ the zero loops of biharmonic polynomials  have more topological restrictions (such as the maximal depth of nests and the total number of nodal loops) compared to general polynomials of the same degree, see Figure \ref{fig:depth-4-5-nodal-loops}.

\subsection*{Acknowledgments} 

This paper involves essential contributions from Bogdan Georgiev, who cannot be named as a coauthor for technical reasons. JGS has been partially supported by NSF under Grants DMS-2245017, DMS-2247537 and DMS-2434314 and by a Simons Fellowship. AL was supported by a Packard Fellowship. JGS is also thankful for the hospitality of the MIT Department of Mathematics, where parts of this paper were done.

\subsection*{LLM Usage Disclosure}

The authors used a workflow combining ChatGPT Pro 5.5, Claude Opus 4.7 and Gemini Pro 3.1 during the development of this work, including for suggesting proof strategies and assisting with calculations. All mathematical statements, proofs, and verifications were subsequently completed and rigorously checked by the authors, who take full responsibility for the results. The writeup of the results was done by the authors.

\section{Proof of the main Theorem}

\subsection{The barrier bump}

Fix a number $0<\eta<1$ throughout the proof. For $R>0$ and an integer $M\ge 1$, define
\begin{align}\label{eq:gamma}
    \Gamma_{R,M}\coloneqq \left\{z=re^{i\theta} \in \C : r^2=R^2\bigl(1+\eta\cos(M\theta)\bigr)\right\}.
\end{align}
By \((r,\theta)\) we will denote the polar coordinates on \(\R^2\). Since $1+\eta\cos(M\theta)>0$, \(\Gamma_{R,M}\) is a smooth star-shaped Jordan curve around the origin.

\begin{figure}[!htb]
    \centering
    \includegraphics[width=.47\textwidth]{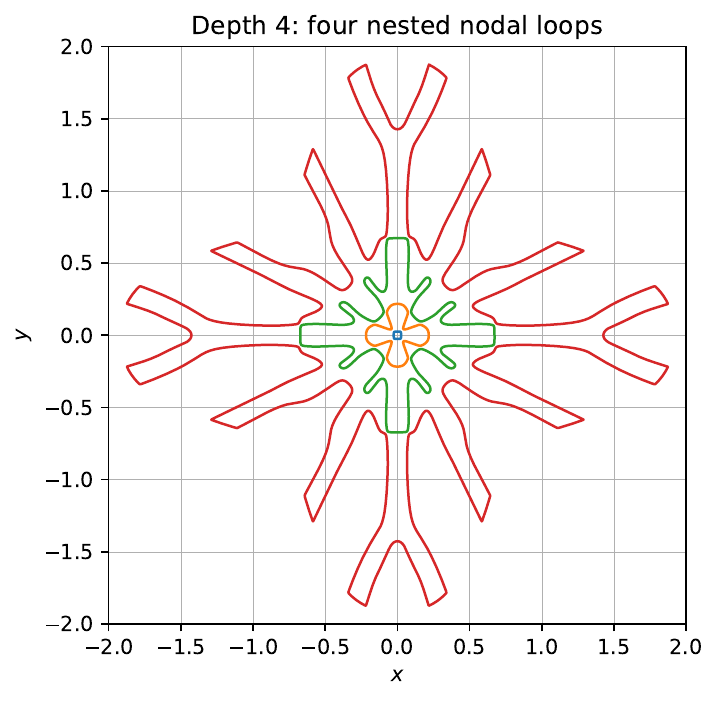}\qquad
    \includegraphics[width=.445\textwidth]{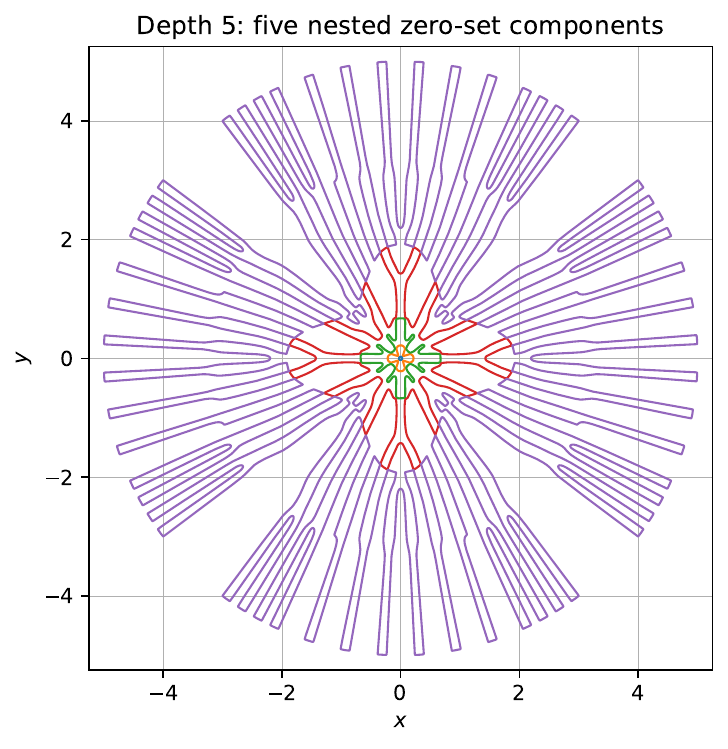}
    \caption{Nested components in zero sets of biharmonic polynomials. The left panel shows four nested smooth Jordan components for an explicit degree-$50$ biharmonic polynomial. The right panel shows five nested components: the inner four are inherited from the degree-$50$ example and the outer loop is added by one localized biharmonic bump. Only the components enclosing the origin are displayed.}
    \label{fig:depth-4-5-nodal-loops}
\end{figure}

Define a polynomial on \(\R^2=\C\)
\begin{align}\label{eq:brm}
    B_{R,M}(z)\coloneqq (R^2-r^2)\Rea(z^M)+\varepsilon_{R,M}\Rea(z^{2M}), \qquad z=re^{i\theta} \in \C, \quad \varepsilon_{R,M}\coloneqq \frac{\eta R^{2-M}}{8(1+\eta)^{M/2}}.
\end{align}  

\begin{lemma}\label{lem:biharmonic-bump}
    The polynomial $B_{R,M}$ is biharmonic.
\end{lemma}

\begin{proof}
    For every $k\ge 1$, $H_k(z)=\Rea(z^k)$ is a homogeneous harmonic polynomial of degree $k$. If $H$ is homogeneous and harmonic of degree $k$, then Euler's identity gives
    \begin{align*}
        \Delta(r^2H)
        =H\Delta(r^2)+2\nabla(r^2)\cdot\nabla H+r^2\Delta H
        =4H+4kH=4(k+1)H.
    \end{align*}
    Therefore $\Delta\Delta(r^2H)=0$. Since $H$ itself is harmonic, every summand in $B_{R,M}$ is biharmonic.
\end{proof}

\begin{lemma}\label{lem:barrier-sign}
    One has \(B_{R,M}(z) <0\) on \(\Gamma_{R,M}.\)
\end{lemma}

\begin{proof}
    Using \(\Rea(z^k)=r^k\cos(k\theta)\) and \(\cos(2M\theta)=2\cos^2(M\theta)-1\) together \(R^2-r^2=-\eta R^2\cos (M\theta)\) on \(\Gamma_{R,M}\), it follows that on \(\Gamma_{R,M}\)
    \begin{align}
        B_{R,M}(z)=-(\eta R^2r^M-2\varepsilon_{R,M}r^{2M})\cos^2(M\theta)-\varepsilon_{R,M} r^{2M}.
    \end{align}
    Since \(r\leq R\sqrt{1+\eta}\) on \(\Gamma_{R,M}\), the definition of \(\varepsilon_{R,M}\) in \eqref{eq:brm} implies that \(\eta R^2r^M-2\varepsilon_{R,M}r^{2M}>0\) from which the claim follows.
\end{proof}

In the next section we will take linear combinations of rescaled bi-harmonic bumps to construct deep nodal nests. See Figure \ref{fig:depth-4-5-nodal-loops} for examples of nestes of depth 4 and 5.

\subsection{Inductive construction} 

Now, we choose $R=1$ and fix $ \eta \in (0,1)$. By the argument above, for every integer $M$ the bi-harmonic polynomial 
\begin{align*}
    P_M(z)\coloneqq  (1-r^2)\Rea(z^M)+  \frac{\eta }{8(1+\eta)^{M/2}} \Rea(z^{2M}), \qquad z=re^{i\theta} \in \C
\end{align*}
is strictly negative on a smooth loop around the origin:  
\begin{align*}
    P_M(z)< 0 \quad \textup{ on }\quad  \Gamma_{M}\coloneqq \left\{z=re^{i\theta}: r^2=1+\eta\cos(M\theta)\right\}.
\end{align*}
With this building block in mind, we will now show by induction that there exists a bi-harmonic polynomial with a nest of $n$ loops around the origin 
such that the polynomial is strictly negative on odd loops and strictly positive on even loops.

\begin{lemma} \label{induction lemma} 
    Let $Q(z)$ be a real polynomial on $\mathbb{R}^2$ with degree  $ \textup{deg}(Q) \leq M-1$ such that $Q<0$ on a compact set $K_{-}$ and $Q>0$ on a compact set $K_+$. Consider polynomials
    \begin{align*}
        Q_\varepsilon(z) \coloneqq \varepsilon^MQ( z\varepsilon^{-1})+P_{M+1}(z), \quad R_\varepsilon(z) \coloneqq \varepsilon^MQ( z\varepsilon^{-1})-P_{M+1}(z).
    \end{align*}
    Then for sufficiently small $\varepsilon>0$
    \begin{align*}
        Q_\varepsilon(z) <  0 \quad \textup{ on } \quad \Gamma_{M+1}, \quad Q_\varepsilon(z)< 0 \quad \textup{ on } \quad \varepsilon K_{-}, \quad Q_\varepsilon(z) > 0 \quad \textup{ on } \quad \varepsilon K_{+}
    \end{align*} 
    and 
    \begin{align*}
        R_\varepsilon(z) >  0 \quad \textup{ on } \quad \Gamma_{M+1}, \quad R_\varepsilon(z)< 0 \quad \textup{ on } \quad \varepsilon K_{-}, \quad R_\varepsilon(z) > 0 \quad \textup{ on } \quad \varepsilon K_{+}.
    \end{align*}
\end{lemma}

\begin{proof}    
    Note that $Q_{\varepsilon}$ converges to $P_{M+1}$ uniformly on $\Gamma_{M+1}$ as $\varepsilon \to 0$, because the degree of $Q$ is strictly smaller than $M$. Hence, for sufficiently small $\varepsilon>0$, the sign of $Q_{\varepsilon}$ on $\Gamma_{M+1}$ is the same as the negative sign of $P_{M+1}$ on $\Gamma_{M+1}$.

    To understand the sign of  $Q_\varepsilon$ on $\varepsilon K_{\pm}$, note that $P_{M+1}(z)=o(|z|^M)$ as $z\to 0$ and consider 
    \begin{align*}
        \varepsilon^{-M} Q_\varepsilon(\varepsilon z) = Q(z)+\varepsilon^{-M} P_{M+1}(\varepsilon z),
    \end{align*}
    which converges uniformly to Q on $K_{\pm}$. Thus the sign $\varepsilon^{-M} Q_\varepsilon(\varepsilon z)$ is positive on $K_+$ and negative on $K_{-}$ for sufficiently small $\varepsilon>0$, which implies that $Q_\varepsilon$ is positive on $\varepsilon K_+$ and negative on $\varepsilon K_{-}$.

    The statement for $R_\varepsilon$ follows similarly. 
\end{proof}

Applying Lemma \ref{induction lemma} for the sequence 
\begin{align*}
    M_{k+1}=2M_k+3, \quad M_1=1,
\end{align*}
we inductively construct  polynomials with the following properties.

\begin{corollary} \label{cor:main}
    For every integer $n$, there is a biharmonic polynomial $u_n$ of degree $2M_n+2$ and $n$ consequently nested disjoint smooth loops $\tilde\Gamma_k$, $k=1,\dots,n$, around the origin such that $u_n$ is strictly negative on the odd loops $\tilde\Gamma_{2k+1}$ and positive on the even loops $\tilde\Gamma_{2k}$. 
\end{corollary}

\subsection{Finishing the proof of Theorem \ref{thm:main}}

Consider the bi-harmonic polynomial $u_{n+1}$ constructed in Corollary \ref{cor:main}, which has a nest of $n+1$ consequently nested loops $\tilde\Gamma_k$, $k=1,\dots,n+1$, such that the sign $u_n$ on $\tilde\Gamma_k$ is $(-1)^k$.

By Sard's theorem the zero set of $u=u_{n+1}+\varepsilon$ is non-singular for generic $\varepsilon >0$. For generic $\varepsilon$ the zero set of $u$ consists of smooth loops and smooth curves that run to infinity.

Choosing a generic and sufficiently small $\varepsilon > 0 $, we still have that the sign of $u$ on $\tilde\Gamma_k$ is $(-1)^k$. Since the loops $\tilde\Gamma_k$ $k=1,\dots,n+1$, where $u$ has alternating sign, are consequently nested, it implies that there are $n$ nested nodal loops of $u$.


\bibliographystyle{amsalpha}
\bibliography{references}

\end{document}